\documentclass{amsart}

\usepackage[pdftex]{graphicx}
\usepackage{pslatex}
\usepackage{latexsym}
\usepackage{amsfonts, amsmath, amssymb, amsthm}
\usepackage{txfonts}
\usepackage{mathrsfs}
\usepackage{mathtools}
\usepackage{changepage}

\newcommand{\real}{\mathbb{R}}

\newcommand{\rmd}{\,\mathrm{d}}

\newcommand{\rmC}{\mathrm{C}}
\newcommand{\rmD}{\mathrm{D}}

\newcommand{\rmB}{\mathrm{B}}
\newcommand{\rmT}{\mathrm{T}}
\newcommand{\rmS}{\mathrm{S}}

\newcommand{\trace}{\mathrm{trace}}

\newcommand{\II}{\mathrm{I}\hspace{-0.8pt}\mathrm{I}}
\newcommand{\st}{\,|\,}

\newtheorem{theorem}{Theorem}[section]
\newtheorem{lemma}[theorem]{Lemma}

\numberwithin{equation}{section}

\begin{document}

\title[HEAT FLOW WITHIN CONVEX SETS]{HEAT FLOW WITHIN CONVEX SETS}
\author{James Dibble}
\address{Department of Mathematics and Statistics, University of Southern Maine, 66 Falmouth Street, Portland, ME 04103}
\email{james.dibble@maine.edu}
\thanks{Some of these results were proved while I was at Rutgers University--New Brunswick and Capital Normal University, during which times I benefitted greatly from conversations with Xiaochun Rong, Penny Smith, and Armin Schikorra. I'm also grateful to Alessandro Goffi for helpful correspondence about the parabolic strong maximum principle.}

\setlength{\topmargin}{1pt}

\maketitle

\begin{abstract}

\noindent A solution to the heat equation between Riemannian manifolds, where the domain is compact and possibly has boundary, will not leave a compact and locally convex set before the image of the boundary does.

\end{abstract}

\section{Introduction}

In their foundational paper \cite{EellsSampson1964}, Eells--Sampson invented the harmonic map heat flow for maps between Riemannian manifolds. If $M$ is a compact Riemannian manifold\footnote{For the sake of simplicity, all manifolds in this paper are assumed to be smooth.} without boundary, $N$ a Riemannian manifold, and $u_0 : M \times \{ 0 \} \rightarrow N$ a $C^1$ map, this flow is a solution $u : M \times [0,\varepsilon) \rightarrow N$ to their heat equation
\begin{equation}\label{heat equation}
    \begin{aligned}
        &\frac{\partial u}{\partial t} = \tau & \textrm{ on } & \phantom{aa} M \times [0,\varepsilon)\\
        &u = u_0 & \textrm{ on } & \phantom{aa} M \times \{ 0 \}
    \end{aligned}
\end{equation}
Here, $\tau$ denotes the tension field of $u$, which is the trace of the second fundamental form of its positive time-slices. Eells--Sampson proved short-term existence and uniqueness of solutions to \eqref{heat equation} for any $\rmC^1$ initial data, as well as long-term existence and uniform subconvergence to harmonic maps when $N$ has nonpositive sectional curvature. This was improved to uniform convergence by Hartman \cite{Hartman1967}. The case where $\partial M \neq \emptyset$ was handled by Hamilton \cite{Hamilton1975}, who proved short-term existence and uniqueness of solutions to the corresponding Dirichlet problem.

\begin{theorem}[Hamilton]\label{short-term existence}
    Let $M$ and $N$ be compact Riemannian manifolds, where $M$ has boundary $\partial M \neq \emptyset$. If $u_0 : M \times \{ 0 \} \rightarrow N$ and $f : \partial M \times [0,T] \rightarrow N$ are smooth maps satisfying $u_0 = f$ on $\partial M \times \{ 0 \}$, then there exists $0 < \varepsilon < T$ such that the heat equation
    \begin{equation}\label{dirichlet heat equation}
        \begin{aligned}
            &\frac{\partial u}{\partial t} = \tau & \textrm{ on } & \phantom{aa} M \times [0,\varepsilon) \setminus \partial M \times \{ 0 \} \\
            &u = u_0 & \textrm{ on } & \phantom{aa} M \times \{ 0 \}\\
            &u = f & \textrm{ on } & \phantom{aa} \partial M \times [0,\varepsilon)
        \end{aligned}
    \end{equation}
    has a unique solution $u : M \times [0,\varepsilon) \rightarrow N$, which is continuous on $M \times [0,\varepsilon)$ and smooth except at the corner $\partial M \times \{ 0 \}$.
\end{theorem}

\noindent Hamilton also proved the long-term existence of solutions to \eqref{dirichlet heat equation}, as well as their uniform convergence to harmonic maps, when $N$ has nonpositive sectional curvature and the boundary data is fixed. A key point in the proof is that, for any solution $u : M \times [0,\varepsilon) \rightarrow N$ to the heat equation, the potential energy density $e = \frac{1}{2} \| \tau \|^2 = \frac{1}{2} g(\tau,\tau)$, where $g$ is the metric on $N$, satisfies
\begin{equation}\label{potential energy density}
    \frac{\partial e}{\partial t} = \Delta e - \| \beta \|^2 - g \big( \mathrm{Ric}_M \nabla_V u, \nabla_V u \big) + g \big( R_N(\nabla_V u, \nabla_W u) \nabla_V u, \nabla_W u \big)\textrm{,}
\end{equation}
where $\Delta$ is the Laplace--Beltrami operator on $M$, $\beta$ is the second fundamental form of the time-slices of $u$, $\mathrm{Ric}_M$ is the Ricci curvature tensor of $M$, $R_N$ is the Riemannian curvature tensor of $N$, and repeated subscripts indicate a trace. The kinetic energy density $\kappa = g \big( \frac{\partial u}{\partial t}, \frac{\partial u}{\partial t} \big)$ satisfies
\begin{equation}\label{kinetic energy density}
    \frac{\partial \kappa}{\partial t} = \Delta \kappa - \big\| \nabla \frac{\partial u}{\partial t} \big\|^2 + g \Big( R_N(\nabla_V u, \frac{\partial u}{\partial t}) \nabla_V u, \frac{\partial u}{\partial t} \Big)\textrm{.}
\end{equation}
Using equations \eqref{potential energy density} and \eqref{kinetic energy density}, along with a number of comparison arguments that build on the parabolic maximum principle, Hamilton was in fact able to show that solutions converge in $C^\infty(M,N)$ to harmonic maps.

Hamilton also proved that solutions to \eqref{heat equation} or \eqref{dirichlet heat equation} will not leave a compact and locally convex subset of $N$ with codimension zero and smooth boundary before the image of $\partial M$ does. That result will be generalized here to arbitrary compact and locally convex sets.

\begin{theorem}\label{maximum principle for heat flow}
    Let $M$ and $N$ be Riemannian manifolds, where $M$ is compact and possibly has boundary $\partial M \neq \emptyset$. Let $Y \subseteq N$ be a compact and locally convex set. Suppose $u : M \times [a,b] \rightarrow N$ is a continuous function that, in the interior of $M \times [a,b]$, is smooth and satisfies $\frac{\partial u}{\partial t} = \tau$. If $u(M \times \{ a \}) \subseteq Y$ and, in the case that $\partial M \neq \emptyset$, $u(\partial M \times [a,b]) \subseteq Y$, then $u(M \times [a,b]) \subseteq Y$.
\end{theorem}

\noindent The proof combines Hamilton's ideas with those of L. Christopher Evans \cite{Evans2010} about viscosity solutions. Specifically, the result is achieved by applying a viscosity form of a maximum principle in \cite{Hamilton1975} to the composition of the flow with the distance function to $Y$.

In \cite{Evans2010}, Evans proved a strong maximum principle for the reaction-diffusion systems in $\real^n$ that he was studying, namely, that a solution that maps into a convex set may touch the boundary at a positive time only if it is entirely contained within the boundary until then. It is reasonable to expect that this carries over to the heat flow on manifolds, although the necessary parabolic strong maximum principle for viscosity solutions appears to be absent in the Riemannian setting. At the end of the paper, a corresponding result for the heat flow is proved, assuming the latter principle.\\

\noindent \textit{Organization of the paper.} Section 2 contains background information about convexity in Riemannian manifolds, Section 3 contains background information about viscosity solutions, and Section 4 contains the proof of Theorem \ref{maximum principle for heat flow}. The proof of a corresponding strong maximum principal for the heat flow is sketched in Section 5, modulo a Riemannian version of the parabolic strong maximum principle for viscosity solutions.

\section{Convexity}

Let $N$ be a Riemannian manifold. A subset $Y \subseteq N$ is \textbf{strongly convex} if, for each $p,q \in Y$, there exists a unique minimal geodesic $\gamma : [0,1] \rightarrow N$ such that $\gamma(0) = p$, $\gamma(1) = q$, and $\gamma([0,1]) \subseteq Y$. The convexity radius of $N$ will be denoted $r : N \rightarrow (0,\infty]$. This is the continuous function characterized by the fact that, for each $y \in N$,
\[
	r(y) = \max \{ \,\varepsilon> 0 \, \,|\, \rmB(x,\delta) \textrm{ is strongly convex for all } 0 < \delta < \varepsilon \}\textrm{.}
\]
A subset $Y \subseteq N$ is \textbf{locally convex} if, for each $p \in \overline{Y}$, there exists $0 < \varepsilon(p) < r(p)$ such that $Y \cap \rmB \big( p,\varepsilon(p) \big)$ is strongly convex.

The following theorem about the structure of locally convex sets was proved by Ozols \cite{Ozols1969} and, independently, Cheeger--Gromoll \cite{CheegerGromoll1972}.

\begin{theorem}\label{locally convex set}

(Ozols, Cheeger--Gromoll) Let $N$ be a Riemannian manifold. If $Y \subseteq N$ is a closed and locally convex set, then $Y$ is an embedded submanifold of $N$ with smooth and totally geodesic interior and possibly non-smooth boundary.

\end{theorem}

\noindent It's also shown in \cite{CheegerGromoll1972} that, at each $p \in Y$, $Y$ has a unique tangent cone, given by
\[
	C_p = \{ t \cdot \exp_p^{-1}(y) \st y \in Y \cap \rmB \big( p,r(p) \big), t \geq 0 \}\textrm{.}
\]
The fact that $C_p$ is a cone means that it is, itself, a convex set. By Theorem \ref{locally convex set}, when $p$ lies in $Y^\circ$, $C_p = \rmT_p Y$. When $p$ lies in $\partial Y$, $C_p$ is a manifold with boundary. In either case, $C_p$ has the same dimension as $Y$.

The following theorem about metric projection onto locally convex sets was proved by Walter \cite{Walter1974}.

\begin{theorem}\label{projection onto locally convex set}

(Walter) Let $N$ be a Riemannian manifold. If $Y \subseteq N$ is a closed and locally convex set, then there exists an open set $U \subseteq N$ containing $Y$ on which the nearest-point projection $\pi : U \rightarrow Y$ is well-defined and locally Lipschitz continuous. Moreover, the map $(x,t) \mapsto \exp_x \big( t \cdot \exp_x^{-1} \pi(x) \big)$, defined from $U \times [0,1]$ into $N$, is a locally Lipschitz strong deformation retraction of $U$ onto $Y$.

\end{theorem}

\noindent If $Y$ in the above theorem is compact, one may take $U = \rmB(Y,\varepsilon)$ for some $\varepsilon > 0$.

\section{Viscosity solutions}

To generalize Hamilton's result to arbitrary convex sets, the idea of viscosity solutions to certain elliptic and parabolic differential equations will be needed. To simplify things, the definitions given here will be rather specific. For a more general treatment of the subject, see \cite{CrandallIshiiLions1992} or, for the elliptic case in the Riemannian setting, \cite{GoffiPediconi2021}.

If $f$ is a real-valued function on a topological space and $x_0$ is in the domain of $f$, then a function $\phi$ \textbf{touches $f$ from above at $x_0$} if $\phi(x_0) = f(x_0)$ and $\phi \geq f$ on a neighborhood of $x_0$. Similarly, $\phi$ \textbf{touches $f$ from below at $x_0$} if $\phi(x_0) = f(x_0)$ and $\phi \leq f$ on a neighborhood of $x_0$.

Let $M$ be a Riemannian manifold, possibly with boundary, $f : M \times [a,b] \to \real$ a continuous function, and $C \in \real$. Then, $f$ is a \textbf{viscosity solution} to $\frac{\partial f}{\partial t} - \Delta f - Cf \leq 0$ at $(x_0,t_0)$ in the interior of $M \times [a,b]$ if, for every $C^2$ function $\phi$ defined on a neighborhood of $(x_0,t_0)$ that touches $f$ from above at $(x_0,t_0)$, one has that $\frac{\partial \phi}{\partial t} - \Delta \phi - C\phi \leq 0$ at $(x_0,t_0)$. Similarly, $f$ is a \textbf{viscosity solution} to $\frac{\partial f}{\partial t} - \Delta f + Cf \geq 0$ at $(x_0,t_0)$ if, for every every $C^2$ function $\phi$ that touches $f$ from below at $(x_0,t_0)$, one has that $\frac{\partial \phi}{\partial t} - \Delta \phi + C\phi \geq 0$ at $(x_0,t_0)$. One defines viscosity solutions $f : M \to \real$ to $\Delta f \leq 0$ and $\Delta f \geq 0$ in the analogous way.

The following maximum principle for viscosity solutions generalizes a result of Hamilton (see p. 101 of \cite{Hamilton1975}).

\begin{theorem}\label{maximum principle}

Let $M$ be a compact Riemannian manifold, possibly with boundary. Suppose that $f : M \times [a,b] \rightarrow \real$ is a continuous function such that $f \leq 0$ on $M \times \{ a \}$ and, in the case that $\partial M \neq \emptyset$, on $\partial M \times [a,b]$. If there exists $C \in \real$ such that, at any point in the interior of $M \times [a,b]$ where $f > 0$, $f$ is a viscosity solution to $\frac{\partial f}{\partial t} - \Delta f - Cf \leq 0$, then $f \leq 0$ on $M \times [a,b]$.

\end{theorem}

\begin{proof}

The trick is to define a function $h : M \times [a,b] \rightarrow \real$ by $h(x,t) = e^{-(C+1)t}f(x,t)$. Then, $h > 0$ if and only if $f > 0$. Fix $a < T < b$. Assume $h$ is positive somewhere on $M \times [a,T]$. Then, by compactness, $h$ achieves a positive maximum on $M \times [a,T]$, say at $(x_0,t_0)$. By assumption, $x_0 \not\in \partial M$ and $0 < t_0 \leq T < b$, so $(x_0,t_0)$ lies in the interior of $M \times [a,b]$, and $h$ satisfies $\frac{\partial h}{\partial t} - \Delta h + h \leq 0$ in the viscosity sense at $(x_0,t_0)$. Because $x_0$ is a global maximum of $h$ on $M \times \{ t_0 \}$, one has that $\Delta h \leq 0$ in the viscosity sense there. Let $\phi$ be a $C^2$ function that touches $h$ from above at $(x_0,t_0)$. By the definition of viscosity solution, $\frac{\partial \phi}{\partial t}|_{(x_0,t_0)} - \Delta \phi|_{(x_0,t_0)} + \phi(x_0,t_0) \leq 0$. Since $\Delta \phi|_{(x_0,t_0)} \leq 0$ and $\phi(x_0,t_0) > 0$, $\frac{\partial \phi}{\partial t}|_{(x_0,t_0)} < 0$. But, this implies that the constant function $\psi(x,t) = h(x_0,t_0)$ touches $h$ from above at $(x_0,t_0)$, which means that $\frac{\partial \psi}{\partial t}|_{(x_0,t_0)} - \Delta \psi|_{(x_0,t_0)} + \psi(x_0,t_0) = \psi(x_0,t_0) = h(x_0,t_0) \leq 0$. This is a contradiction. Thus, $h \leq 0$ on $M \times [a,T]$, and, letting $T \rightarrow b$, the result follows by continuity.

\end{proof}

\section{Proof of Theorem \ref{maximum principle for heat flow}}

Fix everything as in the statement of the theorem. By Theorem \ref{projection onto locally convex set}, there exists $\varepsilon > 0$ such that the projection $\pi : \rmB(Y, \varepsilon) \rightarrow Y$ is well-defined and continuous. Because $\overline{\rmB}(Y,2\varepsilon)$ is compact, standard curvature comparison arguments imply the existence of a lower bound $0 < R < r \big( \overline{\rmB}(Y,\varepsilon) \big)$ for the focal radius $r_f \big( \overline{\rmB}(Y,\varepsilon) \big) = \displaystyle \inf_{y \in \overline{\rmB}(Y,\varepsilon)} r_f(y)$, where by definition
\begin{align*}
r_f(p) = \min \{ T > 0 \st \exists &\textrm{ a non-trivial normal Jacobi field } J \textrm{ along a unit-speed geodesic } \gamma\\
&\textrm{ with } \gamma(0) = p \textrm{, } J(0) = 0 \textrm{, and } \| J \|'(T) = 0 \} \textrm{.}
\end{align*}
Shrinking $\varepsilon$, if necessary, one may suppose that $0 < \varepsilon < R$. Let $H$ be any hyperplane tangent to a point $q \in \overline{\rmB}(Y,\varepsilon)$; that is, let $H$ be an element of the Grassmannian $G \big( n-1,\overline{B}(Y,\varepsilon) \big) \subseteq G(n-1,n)$, where $n = \dim(N)$. Let $v \in H^\perp$ have unit length, so that $H^\perp = \{ tv \st t \in \real \}$. For each $0 \leq t < R$, the exponential map restricted to the normal bundle of the embedded submanifold $\exp_q \big( H \cap \rmB(0,R) \big)$ is a local diffeomorphism around the vector $-t v$. It follows that there exists $0 < \delta_H < R$ such that, for $S_H = \exp_q \big( H \cap \rmB(0,\delta_H) \big)$ and $P_H = \{ w \in S_H^\perp \st \| w \| < \varepsilon \}$, the map $\exp|_{P_H}$ is a diffeomorphism onto its image. Similarly, there exists an open set $U_v$ containing $\exp_q(-tv)$ such that, for each $z \in U_v$, the minimal geodesic $\gamma_z$ connecting $z$ to $q$ remains inside $\exp(P_H)$. Without loss of generality, one may take $U_v$ to be small enough that $U_v \cap S_H = \emptyset$.

Let $\rmS Y = \{ w_y \in \rmT N \st y \in Y, \| y \| = 1 \}$ denote the unit sphere bundle of $Y$. Denote by $\mathscr{S}_{G \big( n-1,\overline{\rmB}(Y,\varepsilon) \big)}$ the space of bilinear forms on hyperplanes in $G \big( n-1,\overline{\rmB}(Y,\varepsilon) \big)$. Define a function $\II : \rmS Y \times [0,\varepsilon] \rightarrow \mathscr{S}_{G \big( n-1,\overline{\rmB}(Y,\varepsilon) \big)}$ by setting $\II(w_y,t)$ equal to the second fundamental form of the level set of $\rmd_{S_{w_y^\perp}}$ through $\exp_y (-tw_y)$; equivalently, $\II(w_y,t)$ is the Hessian of $\rmd_{S_{w_y^\perp}}$ at $\exp_y (-tw_y)$. With respect to the usual smooth structure on $\mathscr{S}_{G(n-1,N)}$ inherited from its structure as a vector bundle over $G(n-1,N)$, the map $\II$ is smooth. Let $\mu : \rmS Y \times [0,\varepsilon] \rightarrow \real$ be the function that takes $(w_y,t)$ to the minimum eigenvalue of $\II(w_y,t)$.

\begin{lemma}\label{lipschitz continuous}

The function $\mu$ is Lipschitz continuous.

\end{lemma}

\begin{proof}

Let $V$ be a open subset of $N$ that's small enough that its closure is compact and admits an orthonormal frame $\{ e_1,\ldots,e_{n-1} \}$. For each $A \in \mathscr{S}_{G(n-1,V)}$ and $1 \leq i,j \leq n-1$, let $\varsigma_{ij}(A) = A(e_i,e_j)$. Write $A_\varsigma = \big[ \varsigma_{ij}(A) \big]_{1 \leq i,j \leq n-1}$. Then the eigenvalues of $A$ are equal to the eigenvalues of $A_\varsigma$. In particular, for the minimum eigenvalue function $\nu$, one has that $\nu(A) = \nu(A_\varsigma)$ on $G(n-1,V)$. Following Hamilton, one computes
\[
|\nu(A) - \nu(B)| = |\nu(A_\varsigma) - \nu(B_\varsigma)| \leq \| A_\varsigma - B_\varsigma \| \leq C \sum_{i,j=1}^{n-1} |\varsigma_{ij}(A) - \varsigma_{ij}(B)|\textrm{,}
\]
where $\| \cdot \|$ denotes the usual matrix norm and the constant $C$ exists because all norms on a finite-dimensional space are equivalent. Since the $\varsigma_{ij}$ vary smoothly, the term on the right is a Lipschitz continuous function on $\mathscr{S}_{G(n-1,V)} \times \mathscr{S}_{G(n-1,V)}$, i.e., $|\varsigma_{ij}(A) - \varsigma_{ij}(B)| \leq D \rmd(A,B)$, that latter distance being measured with respect to the natural metric on $\mathscr{S}_{G(n-1,N)}$. Thus $|\nu(A) - \nu(B)| \leq C D \rmd(A,B)$, and $\nu$ is locally Lipschitz. Since a locally Lipschitz function on a compact set is in fact Lipschitz, the restriction $\nu|_{\mathscr{S}_{G \big( n-1,\overline{\rmB}(Y,\varepsilon) \big)}}$ is Lipschitz. Because $\mu$ can be written as the composition of $\nu|_{\mathscr{S}_{G(n-1,\overline{\rmB}(Y,\varepsilon))}}$ with a smooth function defined on a compact set, $\mu$ is Lipschitz.

\end{proof}

\noindent For any $y \in \rmB(Y, \varepsilon) \setminus Y$, let $\gamma_y : [0,2] \rightarrow \rmB(Y, \varepsilon)$ be the unique minimal geodesic satisfying $\gamma_y(0) = y$ and $\gamma_y(1) = \pi(y)$, and set $v = \gamma_y'(1)$ and $H = H_y = v^\perp$ in the above construction. For simplicity of notation, write $\II_y = \II \big( \frac{-v}{\| v \|} \big)$; since this is a bilinear form on the tangent space to the level set of $\rmd_{S_{H_y}}$ through $y$, it accepts pairs of vectors, i.e., $\II_y = \II_y(\cdot,\cdot)$. Denote by $\varpi_y$ the projection from $T_y N$ onto the tangent space of the level set of $\rmd_{S_{H_y}}$ through $y$. Denote by $u_*$ the spatial derivative of $u$, i.e., the restriction of $\rmD u$ to the tangent space of $M$.

\begin{lemma}\label{coordinate-independent laplacian}

	Suppose $u : M \times [a,b] \rightarrow \rmB(Y,\varepsilon)$ is a continuous function that, in the interior of $M \times [a,b]$, is smooth and satisfies the heat equation $\frac{\partial u}{\partial t} = \tau_u$. Let $(x,t)$ be a point in the interior of $M \times [a,b]$ such that $y = u(x,t) \in \rmB(Y,\varepsilon) \setminus Y$. Then, at $(x,t)$, $\rho = \rmd_{S_{H_y}} \circ u$ satisfies
\[
	\Delta \rho = \frac{\partial \rho}{\partial t} + \trace \big( \II_y(\varpi_y \circ u_*, \varpi_y \circ u_*) \big)\textrm{.}
\]

\end{lemma}

\begin{proof}

On a small neighborhood of $(x,t)$, $u$ remains within $\exp(P_{H_y})$. In any local coordinates $(x_1,\ldots,x_m)$ for $M$ around $x$ and $(y_1,\ldots,y_n)$ for $N$ around $y$, Hamilton computes that
\[
g^{ij} \Big[ \frac{\partial^2 \rmd_{S_{H_y}}}{\partial y^\beta \partial y^\gamma} - \frac{\partial \rmd_{S_{H_y}}}{\partial y^\alpha} \Gamma_{\beta \gamma}^\alpha \Big] \frac{\partial u^\beta}{\partial x^i} \frac{\partial u^\gamma}{\partial x^j} = \Delta \rho - \frac{\partial \rho}{\partial t}\textrm{,}
\]
where $g^{ij}$ denote the coordinates of the inverse of the local expression $g_{ij}$ for the metric on $M$ and $\Gamma_{\beta \gamma}^\alpha$ are the Christoffel symbols of the coordinates on $N$. The matrix $\big[ \sigma_{\beta \gamma} \big]_{1 \leq \beta, \gamma \leq n-1}$, where $\sigma_{\beta \gamma} = \frac{\partial^2 \rmd_{S_{H_y}}}{\partial y^\beta \partial y^\gamma} - \frac{\partial \rmd_{S_{H_y}}}{\partial y^\alpha} \Gamma_{\beta \gamma}^\alpha$, is the coordinate representative of $\II_y$. In exponential normal coordinates for $M$ around $x$, $g^{ij}$ becomes the Kronecker delta; in normal coordinates with respect to $\exp|_{S_{H_y}}$ within $\exp(P_{H_y})$, $\sigma_{n \gamma} = \sigma_{\beta n} = 0$ for all $1 \leq \beta,\gamma \leq n$. Writing $m = \dim(M)$, one has that
\begin{align*}
g^{ij} \Big[ \frac{\partial^2 \rmd_{S_{H_y}}}{\partial y^\beta \partial y^\gamma} - \frac{\partial \rmd_{S_{H_y}}}{\partial y^\alpha} \Gamma_{\beta \gamma}^\alpha \Big] \frac{\partial u^\beta}{\partial x^i} \frac{\partial u^\gamma}{\partial x^j} &= \sum_{i=1}^m \Big[ \frac{\partial^2 \rmd_{S_{H_y}}}{\partial y^\beta \partial y^\gamma} - \frac{\partial \rmd_{S_{H_y}}}{\partial y^\alpha} \Gamma_{\beta \gamma}^\alpha \Big] \frac{\partial u^\beta}{\partial x^i} \frac{\partial u^\gamma}{\partial x^i}\\
&= \sum_{i=1}^m \II_y \Big( \varpi_y \circ u_* \big( \frac{\partial}{\partial x_i} \big), \varpi_y \circ u_* \big( \frac{\partial}{\partial x_i} \big) \Big)\\
&= \trace \big( \II_y(\varpi_y \circ u_*, \varpi_y \circ u_*) \big)\textrm{.}
\end{align*}

\end{proof}

\begin{lemma}\label{supporting hyperplane}

The subspace $H_y$ is a supporting hyperplane to $Y$, i.e., the closure $\overline{C}_{\pi(y)}$ of the tangent cone at $\pi(y)$ is contained in a closed half-space $\overline{H}_y$ with boundary $H_y$. Moreover, $\exp_{\pi(y)}^{-1}(y) \not\in \overline{H}_y$.

\end{lemma}

\begin{proof}

This is an immediate consequence of the first variation formula for arclength.

\end{proof}

\begin{lemma}\label{touches from below}

The function $\rmd_{S_y}$ touches $\rmd_Y$ from below at $y$.

\end{lemma}

\begin{proof}

Because $\exp|_{P_y}$ is a diffeomorphism, $\rmd_{S_y}(y) = \rmd(y,\pi(y)) = \rmd_Y(y)$. By Lemma \ref{supporting hyperplane}, for any $z \in U_y$, the geodesic $\gamma_z$ must hit $S_y$ before it hits $Y$. This shows that $\rmd_{S_y} \leq \rmd_Y$ within $U_y$.

\end{proof}

\begin{proof}[Proof of Theorem \ref{maximum principle for heat flow}]

Let $\rmd_Y : N \rightarrow [0,\infty)$ denote the distance to $Y$. The idea is to show that, wherever the composition $\sigma = \rmd_Y \circ u$ is positive and sufficiently small, it is a viscosity solution to an equation of the form $\frac{\partial \sigma}{\partial t} - \Delta \sigma - C \sigma \leq 0$. The result will then follow from Theorem \ref{maximum principle}.

By Lemma \ref{lipschitz continuous}, $\mu$ is Lipschitz continuous. Let $C_0 \geq 0$ be a Lipschitz constant for $\mu$. Since the exponential image of a hyperplane is totally geodesic at the image of the origin, $\mu(w,0) = 0$ for all $w$. Therefore, $|\mu(w,t)| = |\mu(w,t) - \mu(w,0)| \leq C_0|t - 0| = C_0 t$ and, consequently, $\mu(w,t) \geq -C_0 t$. By compactness, $\| u_* \|$ is bounded above on $M \times [a,b]$ by some $D_0 \geq 0$. Let $C = m D_0 C_0$. For all $y \in \overline{\rmB}(Y,\varepsilon)$ and $v = \gamma_y' \big( \rmd(y,Y) \big)$, one has that
\begin{equation}\label{trace bound}
\trace \big( \II_y(\varpi_y \circ u_*, \varpi_y \circ u_*) \big) \geq mD_0 \mu \Big( \frac{-v}{\| v \|}, \rmd(y,Y) \Big) \geq -m D_0 C_0 \rmd(y,Y) = -C \rmd(y,Y)\textrm{.}
\end{equation}
Assume that, somewhere in $M \times [a,b]$, $u$ maps outside of $Y$. This is equivalent to the statement that $\sigma > 0$ somewhere in $M \times [a,b]$. Because $u(M \times \{ a \}) \subseteq Y$ and $\| u_* \| \leq D_0$, one may, without loss of generality, shrink $b$ so that $u(M \times [a,b]) \subseteq B(Y,\varepsilon)$ and still have that $\sigma > 0$ somewhere in $M \times [a,b]$. Since $\sigma = 0$ on $M \times \{ a \}$ and, in the case that $\partial M \neq \emptyset$, on $\partial M \times [a,b]$, there must exist an interior point $(x_0,t_0)$ of $M \times [a,b]$ such that $\sigma(x_0,t_0) > 0$. By \eqref{trace bound},
\[
	\frac{\partial \rho}{\partial t} = \Delta \rho - \trace \big( \II_y(\varpi_y \circ u_*, \varpi_y \circ u_*) \big) \leq \Delta \rho + C \rho
\]
at $(x_0,t_0)$, so at that point $\rho$ satisfies $\frac{\partial \rho}{\partial t} - \Delta \rho - C \rho \leq 0$. Let $\phi$ be any smooth function that touches $\sigma$ from above at $(x_0,t_0)$. By Lemma \ref{touches from below}, $\rho$ touches $\sigma$ from below at $(x_0,t_0)$, which implies that $\phi$ touches $\rho$ from above at $(x_0,t_0)$. Thus $\frac{\partial \phi}{\partial t} = \frac{\partial \rho}{\partial t}$, $\Delta \phi \geq \Delta \rho$, and $\phi = \rho$ at $(x_0,t_0)$. So $\frac{\partial \phi}{\partial t} - \Delta \phi - C\phi \leq \frac{\partial \rho}{\partial t} - \Delta \rho - C\rho \leq 0$ at $(x_0,t_0)$. This shows that $\sigma$ is a viscosity solution to $\frac{\partial \sigma}{\partial t} - \Delta \sigma - C \sigma \leq 0$ at $(x_0,t_0)$. Theorem \ref{maximum principle} implies that $\sigma = 0$ on $M \times [a,b]$, a contradiction.

\end{proof}

\section{A potential strong maximum principle}

It is reasonable to expect that the strong maximum principle in \cite{Evans2010} may be adapted to the Riemannian heat flow. A proof would seem to depend on the following Riemannian version of the parabolic strong maximum principle for viscosity solutions:

\vspace{2pt}

\begin{adjustwidth}{30pt}{30pt}
	(*) Let $M$ be a compact Riemannian manifold, possibly with boundary. Suppose $f : M \times [a,b] \rightarrow [0,\infty)$ is a continuous function such that, for some $C \in \real$, $f$ is a viscosity solution to
	\[
		\frac{\partial f}{\partial t} - \Delta f + Cf \geq 0
	\]
	in the interior of $M \times [a,b]$. If $f(x_0,t_0) = 0$ for some $(x_0,t_0)$ in $M^\circ \times (a,b]$, then $f = 0$ on $M \times [a,t_0]$.
\end{adjustwidth}

\vspace{2pt}

\noindent Unfortunately, this parabolic Riemannian strong maximum principle does not appear to be written up in the literature (cf. \cite{DaLio2004}, \cite{Evans2010}, \cite{GoffiPediconi2021}). I hope to remedy this gap in a future paper. Assuming (*), one may obtain the following.

\begin{theorem}\label{strong maximum principle for the map itself}
	Suppose (*) holds. Let $M$ and $N$ be Riemannian manifolds, where $M$ is compact and possibly has boundary $\partial M \neq \emptyset$. Let $Y \subseteq N$ be a compact and locally convex set. Suppose $u : M \times [0,T] \rightarrow N$ is a continuous function that, in the interior of $M \times [0,T]$, is smooth and satisfies $\frac{\partial u}{\partial t} = \tau$. Suppose also that $u(M \times [0,T]) \subseteq Y$. If there exists $(x_0,t_0)$ in the interior of $M \times [0,T]$ such that $u(x_0,t_0) \in \partial Y$, then $u(M \times [0,t_0]) \subseteq \partial Y$.
\end{theorem}

\noindent The proof of Theorem \ref{strong maximum principle for the map itself} closely tracks the proof of Theorem \ref{maximum principle for heat flow}, so it will only be sketched. The most significant difference is that, while working from inside $Y$, it is not immediately obvious that the distance to a normal supporting hyperplane to $Y$ locally bounds the distance to $\partial Y$ from above. This is rectified by locally embedding $Y$ near its boundary within submanifolds of $N$ in which $Y$ has codimension zero. These may be obtained in the following way: By the Lebesgue number lemma, there exists $\varepsilon > 0$ such that, for every $y \in \partial Y$, both $B(y,\varepsilon)$ and $Y \cap B(y,\varepsilon)$ are strongly convex. For any $y \in Y^\circ \cap B(\partial Y, \varepsilon)$ and any nearest-point projection $\pi(y) \in \partial Y$, let $U_y \subseteq T_y Y$ be a neighborhood of the segment connecting the origin to $\exp^{-1}_y (\pi(y))$ small enough that $\exp_y|_{U_y}$ is a diffeomorphism onto its image. Then, $V_y = \exp_y(U_y)$ is a smooth submanifold of $N$ that has the same dimension as $Y$. Note that, for the unique minimal geodesic $\gamma_y : [0,1] \to Y$ connecting $y$ to $\pi(y)$, $\gamma_y([0,1]) \subseteq V_y$. Let $H_y = \gamma'(1)^\perp \subseteq T_{\pi(y)} V_y$. Then, $H_y$ is a supporting hyperplane to $Y$ within $V_y$, and, when distances are measured with respect to the induced metric on $V_y$, the exponential image $S_{H_y}$ of a small enough ball around the origin in $H_y$ has the property that $d_{S_{H_y}}$ touches $d_{\partial Y}$ from above at $y$. Similarly, if $y \in \partial Y$, then, within the exponential image of a small ball around the origin in the minimal subspace of $T_y N$ containing the cone $C_y$, the distance $d_{S_{H_y}}$ to the exponential image $S_{H_y}$ of a small ball around the origin in any supporting hyperplane $H_y$ to $C_y$ at the origin touches $d_{\partial Y}$ from above at $y$.

The remainder of the argument proceeds more or less as in the previous section. If $y = u(x,t)$, then, at $(x,t)$, the composition $\rho = d_{S_{H_y}} \circ u$ satisfies an equation of the form
\[
	\nabla \rho = \frac{\partial \rho}{\partial t} + \trace \big( \II_y(\varpi_y \circ u_*, \varpi_y \circ u_*) \big)\textrm{,}
\]
where $\varpi_y$ is projection onto the tangent space of the level set of $d_{S_{H_y}}$ through $y$, $\II_y$ is the second fundamental form of that level set, and $u_*$ is the spatial derivative of $u$. Since the largest eigenvalue of $\II_y$ varies Lipschitz continuously, it follows that there exists $C \geq 0$, independent of $(x,t)$, such that $\rho$ is a viscosity solution to an equation of the form
\[
	\frac{\partial \rho}{\partial t} - \Delta \rho + C\rho \geq 0
\]
at $(x,t)$. The proof is finished by applying (*).

\bibliography{bibliography}
\bibliographystyle{amsplain}

\end{document}